\documentclass[platex,dvipdfmx]{amsart}
\usepackage{amssymb}
\usepackage{braket}
\usepackage{mathrsfs}
\usepackage{ifthen}
\usepackage{graphicx}

\numberwithin{equation}{section}

\usepackage{comment}
\usepackage{mleftright}
\usepackage{hyperref}
\usepackage{cleveref}
 \usepackage[all]{xy}
 \usepackage{amscd}
 \usepackage{amsrefs}
 \usepackage{color}

\theoremstyle{plain}
\newtheorem{thm}{Theorem}[section]
\crefname{thm}{Theorem}{Theorems}
\Crefname{thm}{Theorem}{Theorems}
\newtheorem{prop}[thm]{Proposition}
\crefname{prop}{Proposition}{Propositions}
\Crefname{prop}{Proposition}{Propositions}
\newtheorem{lem}[thm]{Lemma}
\crefname{lem}{Lemma}{Lemmas}
\Crefname{lem}{Lemma}{Lemmas}
\newtheorem{cor}[thm]{Corollary}
\crefname{cor}{Corollary}{Corollaries}
\Crefname{cor}{Corollary}{Corollaries}
\newtheorem {claim}[thm]{Claim}
\crefname{claim}{Claim}{Claims}
\Crefname{claim}{Claim}{Claims}

\crefname{property}{Property}{Properties}
\Crefname{property}{Property}{Properties}

\crefname{problem}{Problem}{Problems}
\Crefname{problem}{Problem}{Problems}

\theoremstyle{definition}

\crefname{defn}{Definition}{Definitions}
\Crefname{defn}{Definition}{Definitions}

\crefname{notation}{Notation}{Notations}
\Crefname{notation}{Notation}{Notations}

\crefname{convention}{Convention}{Conventions}
\Crefname{convention}{Convention}{Conventions}

\crefname{cond}{Condition}{Conditions}
\Crefname{cond}{Condition}{Conditions}

\crefname{assum}{Assumption}{Assumptions}
\Crefname{assum}{Assumption}{Assumptions}

\theoremstyle{remark}
\newtheorem{rem}[thm]{Remark}
\crefname{rem}{Remark}{Remarks}
\Crefname{rem}{Remark}{Remarks}
\newtheorem{ex}[thm]{Example}
\crefname{ex}{Example}{Examples}
\Crefname{ex}{Example}{Examples}

\crefname{section}{Section}{Sections}
\Crefname{section}{Section}{Sections}
\crefname{subsection}{Subsection}{Subsections}
\Crefname{subsection}{Subsection}{Subsections}
\crefname{figure}{Figure}{Figures}
\Crefname{figure}{Figure}{Figures}

\newcommand{\Z}{\mathbb{Z}}

\newcommand{\R}{\mathbb{R}}
\newcommand{\C}{\mathbb{C}}

\def\spinc{\text{Spin}^c}
\def\s{\mathfrak{s}}
\def\Con{\mathcal{C}}
\def\G{\mathcal{G}}
\def\su{\mathfrak{su}}
\def\cD{\mathcal{D}}
\def\F{\mathcal{F}}
\def\Index{\mathrm{Index \,}}
\def\im{\text{Im}}
\title[BF type invariant for $4$-manifolds with contact boundary]{A Bauer--Furuta type refinement of Kronheimer--Mrowka's invariant for 4-manifolds with contact boundary
}

\author{Nobuo Iida}

\address{Graduate School of Mathematical Sciences, the University of Tokyo, 3-8-1 Komaba, Meguro, Tokyo 153-8914, Japan}

\email{iida@ms.u-tokyo.ac.jp}

\date{}

\begin{document}

\maketitle

\begin{abstract}
Kronheimer and Mrowka constructed a variant of Seiberg--Witten invariants for a 4-manifold $X$ with contact  boundary in \cite{KM97}.
Using Furuta's finite dimensional approximation, we refine this invariant in the case $H^1(X, \partial X; \R)=0$.
\end{abstract}

\tableofcontents

\section{Introduction}
In \cite{KM97},  
Kronheimer and Mrowka constructed a variant of the Seiberg--Witten invariant for a compact, connected oriented 4-manifold $X$ on whose boundary a contact structure $\xi$ is given.
Their construction is based on the analysis of the Seiberg--Witten equation on a manifold $X^+$ obtained from $X$ by attaching cone-like ends with almost K\"ahler structure. 
In this paper, we construct a refinement of this invariant by using Furuta's finite dimensional approximation on $X^+$ in the case $H^1(X, \partial X; \R)=0$.
\par
The almost K\"ahler structure determines a canonical $\spinc$ sturecture on $X^+\setminus X$.
Following \cite{KM97}, we denote the set of isomorphism classes of extensions of this canonical $\spinc$ sturecture to $X^+$ by $\spinc(X, \xi)$.
\begin{thm}
Let $X$ be a smooth, compact, connected oriented 4-manifold equipped with a contact structure $\xi$ on its boundary.
Suppose in addition $H^1(X, \partial X; \R)=0$.
Then, for each $\s \in \spinc(X, \xi)$, we can construct a Bauer--Furuta-type stable cohomotopy Seiberg--Witten invariant 
\[
\Psi(X, \xi, \s)\in \pi^{st}_{d(\s)}(S^0) \quad(\text{up to sign})
\]
via finite dimensional approximation on $X^+$, where 
$d(\s)$ is the virtual dimension of the moduli space of solutions to the Seiberg--Witten equation in \cite{KM97}, and $\pi^{st}_d(S^0)$ means the $d$-th stable homotopy group of spheres.
This invariant depends only on $(X, \xi, \s)$.
\qed
\end{thm}
Since $1$ is the only constant gauge transformation, this stable homotopy map is not $S^1$ euqivariant unlike the usual Bauer-Furuta invariant for closed 4-manifolds.
In the situation where $\Psi(X, \xi, \s)$ can be defined, we can recover Kronheimer--Mrowka's invariant $\mathfrak{m}(X, \xi, \s)$ as its mapping degree (Theorem \ref{recover}).
\par
Kronheimer and Mrowka proved a non-vanishing theorem for weak symplectic fillings \cite[Theorem 1.1]{KM97}.
Combining this with the fact that $\Psi(X, \xi, \s)$ recovers Kronheimer--Mrowka's invariant, we obtain a similar non-vanishing theorem.
Recall that a weak symplectic filling $(X, \omega)$ of $(\partial X, \xi)$ determines a $\spinc$ struture  $\s_\omega \in \spinc(X, \xi)$ canonically, as explained in \cite{KM97}, .
 
\begin{thm}
Let $(X, \omega)$ be a weak symplectic filling of 
$(\partial X, \xi)$ with \\$H^1(X, \partial X; \R)=0$ (This includes all Stein fillings $(X, \omega)$ of $(\partial X, \xi)$ as special cases). Then, $\Psi(X, \xi, \s_\omega)\in \pi^{st}_0(S^0) \cong \Z$ is a generator.
\qed
\end{thm}

 Bauer gave a connected sum formula which describes the behavior of  Bauer-Furuta invariants under connected sum operations of closed 4-manifolds (\cite{Bau04}).
In this paper, we examine the behavior of Bauer-Furuta invariants for a connected sum of a closed 4-manifold and a 4-manifold with contact boundary.
For a closed oriented connected 4-manifold $X$ with $b_1(X)=0$ equipped with a $\spinc$ structure $\s$, 
write
\[
\Psi(X, \s) \in \pi^{st, \mathbb{T}}_{d(\s)+1}(S^0)
 \] 
 the Bauer-Furuta invariant of $(X, \s)$ as defined in \cite{BF04} and let
 \[
 \tilde{\Psi}(X, \s) \in \pi^{st}_{d(\s)+1}(S^0)
 \]
 be the image of $\Psi(X, \s)$ under the forgetful map that ignores the $\mathbb{T}$ action, where
\[
d(\s)=\frac{1}{4}(c_1(\s)^2-2\chi(X)-3\sigma(X))
\]
is the virtual dimension of the usual Seiberg--Witten moduli space for closed 4-manifold.

We prove the following connected sum formula in a similar way to \cite{Bau04} Theorem 1.1.
\begin{thm}
Let $(X_1, \s_1)$ be a closed oriented connected 4-manifold with $b_1(X_1)=0$ equipped with a $\spinc$ structure.
Let $(X_2, \xi, \s_2)$ be a compact oriented connected 4-manifold with $H^1(X_2, \partial X_2; \R)=0$ equipped with a contact structure on its boundary and $\s_2\in \spinc(X, \xi)$.
Then, we have
\[
\Psi(X_1\#X_2, \xi, \s_1\# \s_2)=\tilde{\Psi}(X_1, \s_1)\wedge \Psi(X_2, \xi, \s_2)\in \pi^{st}_{d(\s_1)+d(\s_2)+1}.
\]
\end{thm}
With this connected sum formula, 
we can construct a simple example 
whose Kronheimer--Mrowka's invariant is zero but the Bauer--Furuta-type  invariant is non-zero which is given by 
the connected sum of the K3 surface and the standard symplectic filling of $S^3$ by $D^4$ equipped with the $\spinc$ sturucture obtained as the connected sum of their $\spinc$ structures determined by their complex structures.
More generally, we have the following result.
\begin{thm}\label{connected sum 0}
For $i=1, 2$, let $X_i$ be a closed oriented 4-manifold, let $\s_i$ be a $\spinc$ structure on $X_i$, and suppose $(X_i, \s_i),\, (i=1, 2)$ satisfy the following conditions. 
\begin{enumerate}
\item
$X_i$ has an almost complex structure and $\s_i$ is the $\spinc$ structure determined by it.
\item
$b_1(X_i)=0$
\item
$b^+(X_i)\equiv 3\, (mod\, 4)$
\item
The usual Seiberg--Witten invariant for closed 4-manifolds $\mathfrak{m}(X_i, \s_i)$ is odd.
\end{enumerate}
Let $(X_3, \omega)$ be a weak symplectic filling of $(\partial X_3, \xi)$ with $H^1(X_3, \partial X_3; \R)=0$ and let $\s_3=\s_{\omega}\in \spinc(X_3, \xi)$ be the canonical $\spinc$ structure associated with $\omega$.
Then, 
\[
\Psi(X_1 \# X_3, \xi, \s_1\#\s_3) \in \pi^{st}_1(S^0)\cong \Z/2
\]
and
\[
\Psi(X_1 \# X_2 \# X_3, \xi, \s_1\#\s_2\# \s_3) \in \pi^{st}_2(S^0)\cong \Z/2
\]
are nonzero.
\qed
\end{thm}
\begin{ex}

All elliptic surfaces
$E(2n)\, (n=1, 2, \dots)$
equipped with the $\spinc$ structure determined by their K\"ahler sturutures
satisfy theconditions for  $(X_1, \s_1)$ and $(X_2, \s_2)$ in the statement of Theorem \ref{connected sum 0}, and all Stein fillings of contact 3-manifolds satisfy the
conditions for $(X_3, \omega)$ in the statement of Theorem \ref{connected sum 0}.

\end{ex}
In the situation of the above theorem,  Kronheimer--Mrowka's moduli space of $(X_1 \# X_3, \xi, \s_1\#\s_3)$ has formal dimension 1 and the  corresponding moduli space of  $(X_1 \# X_2 \# X_3, \xi, \s_1\#\s_2\# \s_3)$ has formal dimension 2.
Since Kronheimer--Mrowka's invariant is defined to be zero when the virtual dimension is nonzero, Kronheimer--Mrowka's invariants for $(X_1 \# X_3, \xi, \s_1\#\s_3)$, $(X_1 \# X_2 \# X_3, \xi, \s_1\#\s_2\# \s_3)$ are zero.
In fact, we can furthermore show that for any element of $\spinc(X_1 \# X_3, \xi)$ and $\spinc(X_1 \# X_2 \# X_3, \xi)$, Kronheimer--Mrowka's invariant is zero (Proposition \ref{vanish}).
\par
One consequence of the connected sum formula above is that our invariant $\Psi(X, \xi, \s)$ is invariant under the connected sum of rational homology spheres (Theorem \ref{QHS4}). 
\par
As an application of the connected sum formula, we can show the following results on the existence of a connected sum decomposition $X=X_1\#X_2$ for a 4-manifold with contact boundary $X$.
This result can be seen as a contact-boundary version of \cite{Bau04} Corollary 1.2, Corollary1.3 for closed manifolds.
\begin{thm}
Let $(X, \xi)$ be a compact oriented connected 4-manifold with contact boundary satisfying $H^1(X, \partial X; \R)=0$.
Suppose there exists $\s \in \spinc(X, \xi)$ such that $\Psi(X, \xi, \s)\neq 0$.
Moreover, suppose $X$ can be decomposed as a connected sum $X_1\#X_2$ for a closed 4-manifold $X_1$ and a 4-manifold with contact boundary $X_2$.
Then, the following holds.
\begin{enumerate}
\item
Suppose $d(\s)=0$. Then $b^+(X_1)=0$.
\item
Suppose $d(\s)=1$. Then either of the following holds.
\begin{enumerate}
\item
$b^+(X_1)=0$ 
\item
$b^+(X_1)\equiv 3\, (\mathrm{mod} 4)$ and there exist a $\spinc$ structure $\s_1$ on $X_1$ and $\s_2 \in \spinc(X_2, \xi)$ such that $\mathfrak{m}(X_1, \s_1)$ and $\mathfrak{m}(X_2, \xi, \s_2)$ are both odd.
\end{enumerate}
\item
Suppose $d(\s)=2$. Then, $b^+(X_1)\not\equiv 1 (\mathrm{mod} 4)$ holds.
Furthermore, if $b^+(X_1)\neq 0$, either of the following holds.
\begin{enumerate}
\item
$b^+(X_1)\equiv 3 \,(\mathrm{mod} 4)$ and there exists a $\spinc$ structure $\s_1$ on $X_1$ such that $\mathfrak{m}(X_1, \s_1)$ is odd.
\item
$b^+(X_1)$ is even and there exists $\s_2 \in \spinc(X_2, \xi, \s_2)$ such that $\mathfrak{m}(X_2, \xi, \s_2)$ is odd.
\end{enumerate}
\end{enumerate}
\qed
\end{thm}
\par
The most important and new difficulty in constructing our invariant $\Psi(X, \xi, \s)$ is  "to realize the Fredholmness and a global slice at the same time" in the following sense.
In Furuta's finite dimensional approximation, the existence of global slice is crucial. Following Furuta, we would like to express the Seiberg--Witten map with a global slice as a sum of a linear Fredholm operator $L$ and a compact quadratic operator $C$.
Though Kronheimer--Mrowka proved the Fredholmness of the linearized Seiberg--Witten map with a local slice, this slice doesn't seem to be a global slice. 
In order to overcome this difficulty, we give a global slice (Theorem \ref{global slice}) and prove that the linearized Seiberg--Witten map with this global slice is Fredholm (Corollary \ref{Fredholm2}).
We use this operator as $L$, which is not $d^++d^*+D^+_{A_0}$ used in the case of closed 4-manifolds \cite{Fur01}\cite{BF04}.
\par
The construction of this paper is as follows.
In section 2, we describe the geometric and analytical settings to construct our invariant $\Psi(X, \xi, \s)$.
In section 3, we establish the global slice and prove the Fredholmness of the linearized Seiberg--Witten map with this global slice. Then we carry out the finite dimensional approximation following \cite{Fur01}\cite{BF04}.
In section 4 we examine properties of $\Psi(X, \xi, \s)$ such as  the recovery of Kronheimer--Mrowka's invariant and the connected sum formula.
\section*{Acknowledgements}
This paper is a part of the author's master thesis.
I would like to express my deep gratitude to my adviser Mikio Furuta for his helpful suggestions and warm encouragements.
I would also like to thank Masaki Taniguchi for his enormous support.
I am grateful to Hokuto Konno, Nobuhiro Nakamura, Takahiro Oba for helpful comments and suggestions.
I thank Tirasan Khandhawit and Hokuto Konno for sharing me their note on Bauer's connected sum formula. 
Finally, I would like to thank the referee for many helpful comments.
The author was supported by JSPS KAKENHI Grant Number 19J23048 and the Program for Leading Graduate Schools, MEXT, Japan.
\section{Settigns}
Throughout this paper, we suppose that all manifolds and bundles are smooth.
\subsection{Geometric settings}
We describe the settings where we construct our invariant.
These are the same as those in \cite{KM97}2.(iii) except the condition $H^1(X, \partial X; \R)=0$.
\par
Let $X$ be a compact, connected oriented 4-manifold with nonempty boundary.
Suppose in addition that $H^1(X, \partial X; \R)=0$.
Note that this condition implies $\partial X$ is connected.
Let $\xi$ be a contact structure on $\partial X$, compatible with the boundary orientation.
Pick a contact 1-form $\theta$ which is positive on the positively-oriented normal field to $\xi$ and an automorphism $J$ of $\xi$ such that $J^2=-1$ and for any nonzero vector $e$ in $\xi$, $(e, Je)$ is a positive basis for $\xi$.
\par
Then, there exists a unique Riemannian metric $g_1$ satisfying the following conditions(\cite{KM97}):
\begin{enumerate}
\item
$\theta$ has unit length.
\item
$d\theta=2*\theta$
\item
$J$ is an isometry with respect to the restriction of $g_1$ to $\xi$.
\end{enumerate}
We give $\R^{\geq 1}\times \partial X$ an almost K\"ahler structure whose Riemannian metric and symplectic form are given by
\[
g_0=dt^2+t^2g_1, \quad \omega_0=\frac{1}{2}d(t^2\theta)
\]
respectively, where $t$ is the coordinate for $\R^{\geq 1}$.
We call $\R^{\geq 1}\times \partial X$ equipped with this almost K\"ahler structure the \textit{almost K\"ahler cone}, following \cite{MR06}.
This almost K\"ahler structure determines a $\spinc$ structure $\s_0$ and a pair $(A_0, \Phi_0)$ on $\R^{\geq 1}\times \partial X$  as  in Definition 2.2 in \cite{KM97},  where $A_0$ is a $\spinc$ connection of $\s_0$ and $\Phi_0$ is a section of the positive spinor bundle of $\s_0$. 
\par
Fix a $\spinc$ structures $\s$ on $X^+$ equipped with an isomorphism $\s \to \s_0$ on $X^+\setminus X$ and a smooth extension of the pair $(A_0, \Phi_0)$ to $X^+$ which belongs to $\s$.
We use the same notation $(A_0, \Phi_0)$ for these extensions.
\subsection{Weighted Sobolev spaces on manifolds with cone-like ends}
Though weighted Sobolev spaces were not used in \cite{KM07}, we need them in order to do the finite dimensional approximation. 
Weighted Sobolev spaces work effectively to show the compactness of the quadratic part of the Seiberg--Witten equation(See section 3).
\par
Fix a smooth map $\sigma: X^+\to \R^{\geq 0}$
such that the restriction $\sigma|_{\R^{\geq 1}\times \partial X}$ agrees with the $\R^{\geq 1}$ coordinate and $\sigma(X)\subset [0, 1]$.
Let $E$ be a smooth vector bundle equipped with a fiberwise inner product  and a unitary connection $A$ on $X^+$.
Let $C^\infty_0(E)$ denote the set of all smooth sections with compact support.
For $l \in \Z^{\geq 0}$, $\alpha \in \R$, 
define a norm on  $C^\infty_0(E)$ by
\[
\|s\|_{L^2_{l
, A, \alpha}(X^+; E)}
=\left(
\sum^l_{j=0} \int_{X^+}|e^{\alpha \sigma}\nabla^j_A s|^2
\right)^{1/2}
\]
for $s \in C^\infty_0(E)$.
We write $L^2_{k,\alpha}(E)$ for the completion of $C^\infty_0(E)$ with respect to the norm $\|\cdot\|_{L^2_{k, A, \alpha}}$.
In the case of $\alpha=0$, we omit denoting $\alpha$.
\par
We summarize properties of these Sobolev norms used in this paper, which can be shown by standard arguments.

\begin{lem}\label{multiplication and compact embedding}
Let $(E_1, |\cdot|_1, A_1), (E_2, |\cdot|_2, A_2)$ be two normed vector bundles equipped with a unitary connection on $X^+$.
Set $W_n=\sigma^{-1}([2^{n-1}, 2^n]) \subset X^+$.
Denote by $\varphi_n: W_1\to W_n$ the diffeomorphism $(t, y)\mapsto (2^{n-1}t,  y)$.
For $i=1, 2$, suppose isomorphisms
\[
(\varphi^*_n E_i)|_{W_1} \cong  E_i|_{W_1}
\]
are given and
there exist constants $a_1, a_2$ such that
\[
|\varphi_n^*s|_{(t, y)}=2^{a_i n}|s|_{\varphi_n(t, y)}
\]
\[
|\nabla^j \varphi^*_n s|_{(t, y)}=2^{(a_i-j)n}|\nabla^j s|_{\varphi_n(t, y)}
\]
for $s \in \Gamma(E_i)$, where we regard $\varphi_n^*s$, $\nabla^j \varphi^*_n $ as  sections of $E_i|_{W_1}$, $(E_i\otimes (T^*X^+)^{\otimes j})|_{W_1}$ respectively by the isomorphism above.
(For example, if $E_1$ is $\Lambda^k$, $|\cdot|_1$ is the norm induced by the Riemannian metric $g_0$, and $A_1$ is the connection induced by the Levi-Civita connection , then, an isomorphism $(\varphi^*_n E_1)|_{W_1} \cong  E_1|_{W_1}$ can be given by regarding  $W_n=W_1$ as a manifold (but the metrics are different) and $a_1=-k$ satisfies the condition.)
\begin{enumerate}
\item (Multiplication)\\
For $\alpha_1,\alpha_2 \in \R$, $l\in \Z^{>2}, \epsilon \in \R^{>0}$, the multiplication 
\[
L^2_{l, A_1, \alpha_1}(E_1)\times L^2_{l, A_2, \alpha_2}(E_2)\to L^2_{l,A_1\otimes A_2, \alpha_1+\alpha_2-\epsilon}(E\otimes F)
\]
is continuous.
\item (Compact embedding)\\
For $l \in \Z^{\geq 1},
\alpha'<\alpha$,
the inclusion 
\[
L^2_{l, A_1, 
\alpha}(E_1)\to L^2_{l-1, A_1, 
\alpha'}(E_1)
\] 
is compact.
\end{enumerate}
\end{lem}
\begin{proof}
We will show (1).
The argument is a modification of the proof of Theorem 13.2.2 of \cite{KM07}.
Let $s $ be a smooth,  compactly supported section of $E_i$ ($i=1$ or $2$).
By the assumption, 
\[
\begin{split}
\|\nabla^j_{A_i} s\|_{L^2(E_i, W_n)}
&=\left(\int_{W_n}|\nabla^j_{A_i} s|^2 vol\right)^{1/2}\\
&=\left(\int_{W_1}2^{2(a_i-j)}|\nabla^j_{A_i} s|^2 2^4 vol\right)^{1/2}\\
&=2^{(a_i-(j-4/2))n}\|\nabla^j \varphi^*_ns\|_{L^2(E_i; W_1)}
\end{split}
\]
and thus
\[
\|s\|^2_{L^2_l(W_n; E_i)}=\sum^l_{j=0}\|\nabla^j_{A_i} s\|^2_{L^2_l(W_n; E_i)}=\sum^l_{j=0}\{2^{-(j-a_i-4/2)n} \|\nabla^j_{A_i}s\|_{L^2(W_n; E_i)}\}^2 .
\]
This yields
\[
 c^n_{ a_i}\|\varphi^*_n s\|_{L^2_l(W_1; E_i)}\leq \|s\|_{L^2_l(W_n; E_i)}\leq c^n_{a_i}\|\varphi^*_n s\|_{L^2_l(W_1; E_i)}, 
\]
where
\[
c_{a}=2^{a-(l-4/2)}, \quad c'_{a}=2^{a+4/2}.
\]
\par
(1) follows if we can show that there is a constant  $C>0 $ which is independent of $n=0, 1, 2, \cdots$ such that 
\[
\|s_1\otimes s_2\|_{L^2_{l, A_1\otimes A_2, \alpha_1+\alpha_2-\epsilon}(W_n; E_1\otimes E_2)}\leq C \|s_1\|_{L^2_{l, A_1, \alpha_1}(W_n; E_1)} \|s_2\|_{L^2_{l, A_2, \alpha_2}(W_n; E_2)}.
\]
Indeed, 
\[
\begin{split}
\|s_1\otimes s_2\|_{L^2_{l, A_1\otimes A_2, \alpha_1+\alpha_2-\epsilon}(X^+; E_1\otimes E_2)}
&=\left(\sum^\infty_{n=0}\|s_1\otimes s_2\|^2_{L^2_{l,A_1\otimes A_2 , \alpha_1+\alpha_2}(W_n;E_2)}\right)^{1/2}\\
&\leq
C\left(\sum^\infty_{n=0}\|s_1\|^2_{L^2_{l, A_1, \alpha_1}(W_n; E_1)}\|s_2\|^2_{L^2_{l, \alpha_2}(W_n;E_2)}\right)^{1/2}\\
&\leq
C\left(\sum^\infty_{n=0}\|s_1\|_{L^2_{l, A_1, \alpha_1}(W_n; E_1)}\|s_2\|_{L^2_{l, \alpha_2}(W_n;E_2)}\right)\\
&\leq
C\left(\sum^\infty_{n=0}\|s_1\|_{L^2_{l, A_1, \alpha_1}(W_n; E_1)}\right)^{1/2}\left(\sum^\infty_{n=0}\|s_2\|_{L^2_{l, \alpha_2}(W_n;E_2)}\right)^{1/2}\\
&=C\|s_1\|_{L^2_{l, A_1, \alpha_1}(X^+; E_1)}\|s_2\|_{L^2_{l, A_2, \alpha_2}}.
\end{split}
\]
Here, we use the fact that the $l^2$ norm is no bigger than the $l^1$ norm and H\"older's inequality.
We have
\[
\begin{split}
\|s_1\otimes s_2\|_{L^2_l(W_n)}
&\leq(c'_{a_1+a_2})^n \|\varphi^*_n(s_1\otimes s_2)\|_{L^2_l(W_1)}\\
&\leq (c'_{a_1+a_2})^n C_1\|\varphi^*_ns_1\|_{L^2_l(W_1)}\|\varphi^*_n s_2\|_{L^2_l(W_1)}\\
&\leq \frac{(c'_{ a_1+a_2})^nC_1}{c^n_{a_1}c^n_{a_2} }\|\varphi^*_ns_1\|_{L^2_l(W_n)}\|\varphi^*_n s_2\|_{L^2_l(W_n)}\\
\end{split}
\]
where
$C_1$ is a constant of the product estimate $L^2_l\times L^2_l \to L^2_l$ on $W_1$.
Thus, 
\[
\begin{split}
\|s_1\otimes s_2\|_{L^2_{l, \alpha_1+\alpha_2-\epsilon}(W_n)}
&=\|e^{(\alpha_1+\alpha_2-\epsilon)\sigma}(s_1\otimes s_2)\|_{L^2_l{W_n}}\\
&\leq \|e^{-\epsilon \sigma}\|_{C^l(W_n)}\|e^{\alpha_1\sigma}s_1\|_{L^2_l(W_n)}\|e^{\alpha_2\sigma}s_2\|_{L^2_l(W_n)}\\
&\leq \|e^{-\epsilon \sigma}\|_{C^l(W_n)}\frac{(c'_{ a_1+a_2})^nC_1}{c^n_{a_1}c^n_{a_2} }\|e^{\alpha_1\sigma}s_1\|_{L^2_l(W_n)}\|e^{\alpha_2\sigma}s_2\|_{L^2_l(W_n)}\\
&=\|e^{-\epsilon \sigma}\|_{C^l(W_n)}\frac{(c'_{ a_1+a_2})^nC_1}{c^n_{a_1}c^n_{a_2} }\|s_1\|_{L^2_{l, A_1, \alpha_1}(W_n)}\|e^{\alpha_2\sigma}s_2\|_{L^2_{l, A_2, \alpha_2}(W_n)}.\\
\end{split}
\]
Since
\[
\|e^{-\epsilon \sigma}\|_{C^l(W_n)}
\leq \text{const.} e^{-\epsilon {2^{n-1}}}, 
\] 
$\|e^{-\epsilon \sigma}\|_{C^l(W_n)}\frac{(c')^n_{ a_1+a_2}C_1}{c^n_{a_1}c^n_{ a_2} }$ is bounded above by a positive constant independent of $n$. Thus the conclusion of (1) holds.
\par
Next, we will show (2). The arguments below is the same as the proof of the Theorem (3.12) of \cite{Loc87}.
Suppose a sequence $(s_k) \subset L^2_{l, \alpha}$ satisfies
\[
\|s_k\|_{L^2_{l, \alpha}}\leq C_0
\]
for a constant $C_0>0$.
Set $\epsilon:=\alpha-\alpha'$, $E_R:=\sigma^{-1}(\R^\geq{R}), K_R:=\sigma^{-1}(\R^{\leq R})$ for $R>1$.
Then, 
\[
\begin{split}
\|s_k\|_{L^2_{l-1, \alpha'}(E_R)}
&=\left(\sum^{l-1}_{j=0}\|e^{(\alpha-\epsilon)\sigma}\nabla^j s_k\|^2_{L^2(E_R)}\right)^{1/2}\\
&\leq e^{-\epsilon R}\left(\sum^{l-1}_{j=0}\|e^{\alpha\sigma}\nabla^j s_k\|^2_{L^2(E_R)}\right)^{1/2}\\
&=e^{-\epsilon R}\|s_k\|_{L^2_{l-1, \alpha}(E_R)}\\
&\leq C_0e^{-\epsilon R}.
\end{split}
\]
We claim that for any natural number $n$, there exists a subsequence $(s'_k)$ such that for all $i, j$, 
\[
\|s'_i-s'_j\|\leq 2^{-n}.
\]
Indeed, we can take $R>1$ such that
\[
\|s_i-s_j\|_{L^2_{l-1, \alpha'}(E_R)}\leq 2^{-n}/\sqrt{2}
\]
by the inequality shown above.
On the other hand, since $K_R$ is compact, we can take a subsequence $(s'_k)$ such that
\[
\|s'_i-s'_j\|_{L^2_{l-1, \alpha'}(K_R)}\leq 2^{-n}/\sqrt{2} 
\]
for all $i, j$ by Rellich's lemma.
Thus, 
\[
\|s'_i-s'_j\|_{L^2_{l-1, \alpha'}(X^+)}=\sqrt{\|s'_i-s'_j\|^2_{L^2_{l-1, \alpha'}(E_R)}+\|s'_i-s'_j\|_{L^2_{l-1, \alpha'}(K_R)}}\leq 2^{-n}
\]
holds.
This shows the claim.
Thus, by diagonal argument, we can take a subsequence  $(s''_k)$ of $(s_k)$ such that
\[
\|s''_i-s''_j\|_{L^2_{l-1, \alpha'}(X^+)}\leq 2^{-\min\{i, j\}}
\]
\end{proof}
\subsection{Analytic settings}
Suppose $l\in \Z^{\geq 4}$ and $\alpha \geq 0$.
Define 
\[
\Con_{l,\alpha}=\{(A,\Phi) | A-A_0 \in L^2_{l,
\alpha}, \text{and}\, \Phi-\Phi_0\in L^2_{l,
{\alpha}, A_0}\}
\]
\[
\G_{l+1,\alpha}=\{ X^+ \to  \C | |u|=1, \text{and}\, 1-u\in L^2_{l+1,
{\alpha}}\}.
\]
Then $\G_{l+1,
{\alpha}}$ smoothly acts on $\Con_{l,
{\alpha}}$  by 
\[
u\cdot(A, \Phi)=(A-u^{-1}du, u\Phi)
\]
as usual.
For $(A, \Phi)\in \Con_{l,
{\alpha}}$, consider the Seiberg--Witten equation on $X^+$ as in \cite{KM97}:
\begin{equation}\label{SWeq}
\begin{split}
\frac{1}{2}F^+_{A^t}-\rho^{-1}(\Phi\Phi^*)_0&=\frac{1}{2}F^+_{A^t_0}-\rho^{-1}(\Phi_0\Phi^*_0)_0\\
D^+_{A}\Phi&=0
\end{split}
\end{equation}
where  $\rho : i\Lambda^+ \to i\su(S^+)$ is the isomorphism given by the Clifford multiplication.

We denote the map corresponding to the equation by 
\[
\begin{split}
\mathcal{F}: \Con_{l,
{\alpha}}&\to L^2_{l-1,
{\alpha}}(i\Lambda^+ \oplus S^-)\\
(A, \Phi)&\mapsto \left(\frac{1}{2}F^+_{A^t}-\rho^{-1}(\Phi\Phi^*)_0-\left( \frac{1}{2}F^+_{A^t_0}-\rho^{-1}(\Phi_0\Phi^*_0)_0\right), D^+_A \Phi\right).
\end{split}
\].
Note that by the product estimate of Lemma {multiplication and compact embedding}, this map is well-defined.
Indeed, if we write $A=A_0+a$, $\Phi=\Phi_0+\phi$, we have
\[
\frac{1}{2}F^+_{A^t}-\rho^{-1}(\Phi\Phi^*)_0-\left( \frac{1}{2}F^+_{A^t_0}-\rho^{-1}(\Phi_0\Phi^*_0)_0\right)
=d^+a -\rho^{-1}(\phi\phi^*)_0-\rho^{-1}(\Phi_0\phi^*+\phi \Phi^*_0)
\]
\[
D^+_A\Phi=D^+_{A_0}\Phi_0+D^+_{A_0}\phi+\rho(a)\Phi_0+\rho(a)\phi.
\]
and
 $D^+_{A_0}\Phi_0$ is zero on the cone by the definition of $(A_0, \Phi_0)$.

\section{Construction}
In this section, we do the finite dimensional approximation of the Seiberg-Witten equation with a slice, and construct our Bauer--Furuta-type invariant.
The construction is similar to that of \cite{Fur01}\cite{BF04}; that is, we decompose Seiberg--Witten map with a global slice into the sum $L+C: V\to W$, where $L$ is a linear Fredholm map and $C$ is a quadratic compact map and make the finite dimensional approximation of it, but we have to modify the following three points.
\begin{enumerate}
\item
Since $X^+$ is noncompact, we can't apply Rellich's theorem.
In order to ensure the compactness of the quadratic term $C$,  we introduce weighted Sobolev spaces.
\item
As we will show in Proposition \ref{Fredholmness 1}, we can directly deduce the Fredholmness of $L'=\mathcal{D}_{(A_0, \Phi_0)}\F+\delta_{(A_0, \Phi_0)}^{
{*, \alpha}}$ from the Fredholm theory in \cite{KM97} when $\alpha>0$ is sufficiently small.  (where $\delta_{(A_0, \Phi_0)}^{*, \alpha}$ is the $L^2_\alpha$ formal adjoint of infinitesimal gauge action at $(A_0, \Phi_0)$).
Our construction is based on this result.
This is the reason why we don't use the operator $d^++d^*+D^+_{A_0}$, which is used in the case of closed manifolds \cite{Fur01}\cite{BF04} as the linear part, since the author doesn't know whether it is Fredholm or not in our setting. 
\item
In order to do the finite dimensional approximation, we need a global slice.
The local slice  $\ker \delta_{(A_0, \Phi_0)}^{
{*, \alpha}}$ doesn't seem to give a global slice, so we use instead 
$\ker d^{*, \alpha}$.
Then, we have to show two things.
First, we must show that  $d^{
{*, \alpha}}$ actually gives a global slice in the sense of Theorem \ref{global slice} below.
Second,  We have to show the Fredholmness of $L=\mathcal{D}_{(A_0, \Phi_0)}\F+d^{
{*, \alpha}}$. 
\end{enumerate}
\subsection{Fredholmness of the linearized Seiberg-Witten map with a local slice}
In this subsection, we see an immediate consequence of the Fredholm theory of \cite{KM97}, on which our argument is based.
\par
Suppose $\alpha \geq 0$.
Denote the infinitesimal gauge action at $(A_0, \Phi_0)$ by
\[
\begin{split}
\delta_{(A_0, \Phi_0)}: i\Omega^0(X^+)&\to \Gamma(X^+; i\Lambda^1 \oplus S^+)\\
f&\mapsto (-df, f\Phi_0)
\end{split}
 \]
 and its $L^2_\alpha$ formal adjoint by
 \[
 \begin{split}
\delta^{
{*, \alpha}}_{(A_0, \Phi_0)}: \Gamma(X^+; i\Lambda^1 \oplus S^+)&\to i\Omega^0(X^+)\\
(a, \phi)&\mapsto -d^{
{*, \alpha}}a+i\text{Re}\langle i\Phi_0, \phi\rangle.
\end{split}
 \]
 For $k \in \Z^{\geq 1}$, define 
 \[
L':  L^2_{k,
{\alpha}}(i\Lambda^1 \oplus S^+) \to L^2_{k-1,
{\alpha}}(i\Lambda^{2+} \oplus S^+\oplus i\Lambda^0)
\]
by
\[
L'
\begin{bmatrix}
a\\
\phi
\end{bmatrix}
=
\begin{bmatrix}
\cD_{(A_0, \Phi_0)}\mathcal{F}(a, \phi)
\\
\delta^{
{*, \alpha}}_{(A_0, \Phi_0)}(a, \phi).
\end{bmatrix}
\]
\begin{prop}\label{Fredholmness 1}
For $k\in \Z^{\geq1}$ and sufficiently small $\alpha>0$, $L'$ is Fredholm and its index is given by
\[
\Index{L'}=d(\s)=\langle e(S^+, \Phi_0), [X, \partial X]\rangle.
\]
in terms of the relative Euler class of $S^+$ with respect to the section $\Phi_0$.
\end{prop}
\begin{proof}
In the case $\alpha=0$, the result is proved in \cite{KM97} Theorem 3.3.
Here, in order to make clearfy the dependance of $L'$ on $\alpha$, we denote it by $L'_\alpha$.
It can be written explicitly as follows:
\[
L'_\alpha
\begin{bmatrix}
a\\
\phi
\end{bmatrix}
=
\begin{bmatrix}
d^+a-\rho^{-1}(\Phi_0\phi^*+\phi\Phi^*_0)_0\\
D^+_{A_0}\phi+\rho(a)\Phi_0\\
-d^{
{*, \alpha}}a+i\text{Re}\langle i\Phi_0, \phi\rangle
\end{bmatrix}
\]
By the definition of $L^2_\alpha$ inner product,  for a 0-form $f$  and a 1-form $\eta$ on $X^+$ which are smooth and have compact support, 
\[
\langle df, \eta \rangle_{L^2_\alpha}
=\int_{X^+}e^{2\alpha \sigma}\langle df, \eta\rangle
=\int_{X^+}e^{2\alpha \sigma}\langle f,  e^{-2\alpha \sigma}d^*e^{2\alpha \sigma}\eta\rangle
=\langle f, e^{-2\alpha \sigma}d^*e^{2\alpha \sigma} \eta \rangle_{L^2_\alpha}
\]
holds, so we have
\[
d^{*, \alpha}=e^{-2\alpha\sigma}d^*e^{2\alpha \sigma}.
\]
Consider the family of operators 
\[
e^{\alpha \sigma}\circ L'_\alpha \circ e^{-\alpha\sigma}: L^2_l \to L^2_{l-1}.
\]
This can be expressed explicitly as
\[
e^{\alpha \sigma}\circ L'_\alpha \circ e^{-\alpha\sigma}
\begin{bmatrix}
a\\
\phi
\end{bmatrix}
=
\begin{bmatrix}
d^+a-\alpha (d\sigma\wedge a)^+ -\rho^{-1}(\Phi_0\phi^*+\phi\Phi^*_0)_0\\
D^+_{A_0}\phi-\alpha \rho(d\sigma)\phi +\rho(a)\Phi_0\\
-d^*a+\alpha *^{-1}(d\sigma\wedge  *a)+i\text{Re}\langle i\Phi_0, \phi\rangle
\end{bmatrix}, 
\]
since we have
\[
 e^{\alpha \sigma}d^+e^{-\alpha\sigma}a=d^+a-\alpha (d\sigma\wedge a)^+
\]
\[
 e^{\alpha \sigma}D^+_{A_0}e^{-\alpha\sigma}\phi=D^+_{A_0}\phi-\alpha\rho(d\sigma)\phi
\]
\[
e^{\alpha \sigma}d^{*, \alpha}e^{-\alpha\sigma}a=e^{-\alpha \sigma}d^{*}e^{\alpha\sigma}a=-e^{-\alpha \sigma}*^{-1}d*e^{\alpha\sigma}a=d^*a-\alpha *^{-1}(d\sigma\wedge *a).
\]
Thus, $e^{\alpha \sigma}\circ L'_\alpha \circ e^{-\alpha\sigma}$ depends on $\alpha$ continuous way.

The fact that the multiplication by $e^{\alpha \sigma}$ is an isomorphism from $L^2_{l, \alpha}$ to $L^2_l$ and the Fredholmness being an open condition imply the conlusion.
\end{proof}
\subsection{The global slice theorem}
In this section, we shall show the global slice theorem, which claims that taking intersection with $i\ker (d^{
{*, \alpha}}: L^2_{l,
{\alpha}}(X^+; \Lambda^1)\to L^2_{l-1,
{\alpha}}(\Lambda^0)) \oplus L^2_{l,
{\alpha}}(X^+; S^+) \subset L^2_{l,
{\alpha}}(i\Lambda^1\oplus S^+)$ is effectively the same as taking the quotient by the gauge group. Later, this result is used to identify the inverse image of a certain point under $L+C$ with the moduli space and deduce the compactness of the former from Kronheimer--Mrowka's compactness result for the latter.
\par
\begin{lem}\label{decomposition}
For $l \in \Z^{\geq 1}$ and $\alpha>0$ enough small, we have a $L^2_\alpha$ orthogonal decomposition
\[
L^2_{l,
{\alpha}}(X^+; \Lambda^1)=\im (d: L^2_{l+1,
{\alpha}}(\Lambda^0)\to L^2_{l,
{\alpha}}(\Lambda^1))\oplus \ker (d^{
{*, \alpha}}: L^2_{l,
\alpha}(\Lambda^1)\to L^2_{l-1,
{\alpha}}(\Lambda^0)).
\]
Furthermore, each summand is a closed subspace.
\end{lem}
\begin{proof} 
We write $E=\Lambda^0\oplus \Lambda^+$, $F=\Lambda^1$.
We write
\[
D=d\oplus d^{
{*, \alpha}}: \Gamma(X^+; E)\to \Gamma(X^+; F)
 \]
and
\[
D^{
{*, \alpha}}=d^{
{*, \alpha}}+d^+: \Gamma(X^+; F)\to \Gamma(X^+; E).
 \]
These operators are the formal $L^2_\alpha$ adjoints of each other.
\begin{claim}
For $l \in \Z^{\geq 1}$ and any sufficiently small positive $\alpha$, the images of
\[
D, D^{
{*, \alpha}}: L^2_{l,
{\alpha}}\to L^2_{l-1,
{\alpha}}
\]
are closed subspaces respectively.
\end{claim}
\underline{Proof of Claim.}
In the previous section, we showed
\[
L'
\begin{bmatrix}
a\\
\phi
\end{bmatrix}
=
\begin{bmatrix}
d^+a-\rho^{-1}(\Phi_0\phi^*+\phi\Phi^*_0)_0\\
D^+_{A_0}\phi+\rho(a)\Phi_0\\
-d^{
{*, \alpha}}a+i\text{Re}\langle i\Phi_0, \phi\rangle
\end{bmatrix}
\]
is a Fredholm map from $L^2_{l,
{\alpha}}$ for any sufficiently small $\alpha>0$.
Since the image of a closed subspace under a bounded Fredholm operator is closed, $i \im D^{*, \alpha}=L'(L^2_{l, \alpha}(i\Lambda^1 \oplus 0))\cap L^2_{l-1,
{\alpha}}(i\Lambda^+ \oplus i\Lambda^0)\subset L^2_{l-1,
{\alpha}}(E)$ is closed.
This shows the claim about $D^{*, \alpha}$.
In \cite{KM97}Theorem 3.3, Kronheimer and Mrowka also proved the Fredholmness of the operator
\[
\begin{split}
L^2_l(i\Lambda^0\oplus i\Lambda^+\oplus S^-)&\to L^2_{l-1}(i\Lambda^1\oplus S^+)\\
(f, \eta, \psi)&\mapsto (-df+d^*\eta+i\text{Re}\langle i\Phi_0, \psi\rangle, f\Phi_0+D^-_{A_0}\psi-\rho(\eta)\Phi_0).
\end{split}
\]
A similar argument shows the claim about $D$. 
\qed
\par
\begin{claim}
For $l\in \Z^{\geq 1}$ and any sufficiently small positive $\alpha$, we have a $L^2_\alpha$ orthogonal decomposition
\[
L^2_{l,
{\alpha}}(F)=\im (D: L^2_{l+1,
{\alpha}}(E)\to L^2_{l,
{\alpha}}(F))\oplus \ker (D^{
{*, \alpha}}: L^2_{l,
{\alpha}}(F)\to L^2_{l-1,
{\alpha}}(E)).
\]
\end{claim}
\underline{Proof of Claim.}
By taking $l=1$ in the previous claim, we have a $L^2_\alpha$ orthogonal decomposition
\[
L^2_{\alpha}(F)=\im (D: L^2_{1,
{\alpha}}(E)\to L^2_{\alpha}(F))\oplus  (\im (D: L^2_{1,
{\alpha}}(E)\to L^2_{\alpha}(F)))^{\perp_{L^2_\alpha}}
\]
By the elliptic regularity, we have
\[
(\im (D: L^2_{1,
{\alpha}}(E)\to L^2_{\alpha}(F)))^{\perp_{L^2_\alpha}}=\ker(D^{
{*, \alpha}}: (L^2_{\alpha}\cap C^\infty)(F)\to C^\infty(E)).
\]
Thus, we have a $L^2_\alpha$ orthogonal decomposition
\[
L^2_{\alpha}(F)=\im (D: L^2_{1,
{\alpha}}(E)\to L^2_{\alpha}(F))\oplus  \ker(D^{
{*, \alpha}}: (L^2_{\alpha}\cap C^\infty)(F)\to C^\infty(E)).
\]
Next, we will derive the desired decomposition from this.
$\im (D: L^2_{l+1,
{\alpha}}(E)\to L^2_{l,
{\alpha}}(F))\cap \ker (D^{
{*, \alpha}}: L^2_{l,
{\alpha}}(F)\to L^2_{l-1,
{\alpha}}(E))=0$ is obvious.
Thus, it is sufficient to show that
\[
L^2_{l,
{\alpha}}(F)=\im (D: L^2_{l+1,
{\alpha}}(E)\to L^2_{l,
{\alpha}}(F))+ \ker (D^{
{*, \alpha}}: L^2_{l,
{\alpha}}(F)\to L^2_{l-1,
{\alpha}}(E)).
\]
Let  $x$ be an element of $ L^2_{l,
{\alpha}}(F)$ .
Since $L^2_{l,
{\alpha}} $ is contained in $ L^2_{\alpha}$, $x$ can be uniquely expressed as a sum
\[
x= D y+z, \quad y \in L^2_{1,
{\alpha}}(E), \, z \in \ker(D^{
{*, \alpha}}: (L^2_{\alpha}\cap C^\infty)(F)\to C^\infty(E)).
\]
by the decomposition we have shown above.
By the elliptic regularity of $D^{
{*, \alpha}}$, we have $z \in \ker (D^{
{*, \alpha}}: L^2_{l,
{\alpha}}(F)\to L^2_{l-1,
{\alpha}}(E))$.
In turn, the equation 
\[
D y=x-z \in L^2_{l,
{\alpha}}(F)
\]
and the elliptic regularity of $D$ imply $y \in L^2_{l+1,
{\alpha}}(E)$.
This completes the claim.
\qed
\par
\underline{Proof of Lemma.}
It is clear that
\[
\im(D: L^2_{l+1,
{\alpha}}(E))=\im (d: L^2_{l+1,
{\alpha}}(\Lambda^0)\to L^2_{l,
{\alpha}}(\Lambda^1))\oplus \im (d^{
{*, \alpha}}: L^2_{l+1,
{\alpha}}(\Lambda^+)\to L^2_{l,
{\alpha}}(\Lambda^1)).
\]
Now, we have $L^2_\alpha$ orthogonal decomposition
\[
\ker (d^{
{*, \alpha}}: L^2_{l,
{\alpha}}(\Lambda^1)\to L^2_{l-1,
{\alpha}}(\Lambda^0))=\im (d^{
{*, \alpha}}: L^2_{l+1,
{\alpha}}(\Lambda^1)\to L^2_{l,
{\alpha}}(\Lambda^0))\oplus \ker(D^{
{*, \alpha}}: (L^2_{l,
{\alpha}}(F)\to L^2_{l-1,
{\alpha}}(E)).
\]
Indeed, any $a \in \ker (d^{
{*, \alpha}}: L^2_{l,
{\alpha}}(\Lambda^1)\to L^2_{l-1,
{\alpha}}(\Lambda^0))$ can be expressed as
\[
a=db+d^{*, \alpha}c+h , \quad b \in L^2_{l+1, \alpha}(\Lambda^0), \, c \in L^2_{l+1, \alpha}(\Lambda^+), \, h \in  \ker(D^{
{*, \alpha}}: (L^2_{l,
{\alpha}}(F)\to L^2_{l-1,
{\alpha}}(E))
\] 
by the previous claim.
and it is sufficient to show $db=0$. This is true since
\[
\|db\|^2_{L^2_{\alpha}}=\langle b, d^{*, \alpha}db\rangle_{L^2_\alpha}=\langle b, d^{*, \alpha}(a-d^{*, \alpha}c-h)\rangle_{L^2_\alpha}=0.
\]
\par
Thus, using the previous claim again, we have
\[
\begin{split}
&L^2_{l,
{\alpha}}(F)\\=&\im (D: L^2_{l+1,
{\alpha}}(E)\to L^2_{l,
{\alpha}}(F))\oplus \ker (D^{
{*, \alpha}}: L^2_{l,
{\alpha}}(F)\to L^2_{l-1,
{\alpha}}(E))\\
=&\im (d: L^2_{l+1,
{\alpha}}(\Lambda^0)\to L^2_{l,
{\alpha}}(\Lambda^1))\oplus \im (d^{
{*, \alpha}}: L^2_{l+1,
{\alpha}}(\Lambda^+)\to L^2_{l,
{\alpha}}(\Lambda^1))\\
&\oplus \ker (D^{
{*, \alpha}}: L^2_{l,
{\alpha}}(F)\to L^2_{l-1,
{\alpha}}(E))\\
=&\im (d: L^2_{l+1,
{\alpha}}(\Lambda^0)\to L^2_{l,
{\alpha}}(\Lambda^1))\oplus \ker (d^{
{*, \alpha}}: L^2_{l,
{\alpha}}(\Lambda^1)\to L^2_{l-1,
{\alpha}}(\Lambda^0))
\end{split}
\]
as claimed.
\end{proof}
The following is the global slice theorem that we need in the finite dimensional approximation.
We use the condition $H^1(X, \partial X; \R)=0$ here.
\begin{prop}\label{global slice}(The global slice theorem)\,
For any sufficiently small positive $\alpha$, the map
\[
\begin{split}
((A_0, \Phi_0)+(i\ker(d^{
{*, \alpha}}: L^2_{l,
{\alpha}}(\Lambda^1)\to L^2_{l-1,
{\alpha}}(\Lambda^0)))\oplus L^2_{l,
{\alpha}}(S^+))\times \G_{l+1,
{\alpha}}
&\to \Con_{l,
{\alpha}}\\
(A, \Phi)&\mapsto u\cdot(A, \Phi)
\end{split}
\]
is a diffeomorphism.
\end{prop}
\begin{proof}
It is enough to show the map
\[
\begin{split}
\varphi: i\ker(d^{
{*, \alpha}}: L^2_{l,
{\alpha}}(\Lambda^1)\to L^2_{l-1,
{\alpha}}(\Lambda^0))\times 
\G_{l+1,
{\alpha}}
&\to L^2_{l,
{\alpha}}(i \Lambda^1)\\
(a, u)&\mapsto a-u^{-1}du
\end{split}
\]
is a diffeomorphism.
We will show that $\varphi$ is bijective and its derivative at any point is an isomorphism, then the conclusion holds by the inverse function theorem.
The surjectivity of $\varphi$ follows from the decomposition of Lemma \ref{decomposition}.
We will prove the injectivity of $\varphi$. 
\begin{claim}
For any sufficiently small positive $\alpha$, we have
\[
\ker(d^{
{*, \alpha}}: L^2_{l,
{\alpha}}(\Lambda^1)\to L^2_{l-1,
{\alpha}}(\Lambda^0)) \cap \ker (d: L^2_{l,
{\alpha}}(\Lambda^1)\to L^2_{l-1,
{\alpha}}(\Lambda^2)) =0.
 \]
\end{claim}
\underline{Proof of Claim.}
By \cite{HHM04} Theorem 1A, and the assumption $H^1(X, \partial X; \R)=0$, the claim follows in the case $\alpha=0$.
For small positive $\alpha$, the claim follows since the dimension of the kernel of $d+d^{
{*, \alpha}}$ is upper half continuous as we vary $\alpha$.
\qed
\par
In order to prove the injectivity of $\varphi$, it is enough to show $a=0$ and $u=1$ provided that $\varphi(a, u)=a-u^{-1}du=0$.
By assumption, $da=d^{
{*, \alpha}}a =0$, thus, $a=0$ by the above claim.
This also means that $u^{-1}du=0$, and since $1-u \in L^2_{l+1,
{\alpha}}$, $u$ must be identically $1$ on $X^+$.
\par
The fact that 
$\cD_{(a, u)}\varphi$ is isomorphic for any $(a, u)\in i\ker d^{
{*, \alpha}}\times \G_{l+1,
{\alpha}}$ also follows from the decomposition of Lemma \ref{decomposition}.
\end{proof}
As a corollary of the global slice theorem above, we can show the following Fredholmness result which we need in the finite dimensional approximation.
We write 
$V=L^2_{l,
{\alpha}}(i\Lambda^1\oplus S^+), W=L^2_{l-1,
{\alpha}}(i\Lambda^0\oplus i\Lambda^+\oplus S^-)$ and
define $L: V \to W$ by
\[
L
\begin{bmatrix}
a\\
\phi
\end{bmatrix}
=
\begin{bmatrix}
\cD_{(A_0, \Phi_0)}\F(a, \phi)\\
d^{*, \alpha}a
\end{bmatrix}.
\]

\begin{cor}\label{Fredholm2}
For any sufficiently small positive $\alpha$, $L$ is Fredholm and its index is the same as that of $L'$.
\end{cor}
\begin{proof}
We denote $d^{
{*, \alpha}}:  L^2_{l,
{\alpha}}(\Lambda^1)\to L^2_{l-1,
{\alpha}}(\Lambda^0)$ and 
$\delta^{
{*, \alpha}}_{(A_0, \Phi_0)}: L^2_{l,
{\alpha}}(i\Lambda^1 \oplus S^+)\to L^2_{l-1,
{\alpha}}(i\Lambda^0)$ just by $d^{
{*, \alpha}}$ and $\delta^{
{*, \alpha}}_{(A_0, \Phi_0)}$ respectively.
The Proposition \ref{global slice} implies that the inclusion 
\[
i\ker d^{
{*, \alpha}}\oplus L^2_{l,
{\alpha}}(S^+)
\hookrightarrow T_{(A_0, \Phi_0)}\Con_{l,
{\alpha}}
\]
induces a linear isomorphism
\[
i\ker d^{
{*, \alpha}}\oplus L^2_{l,
{\alpha}}(S^+)\to T_{(A_0, \Phi_0)}\Con_{l,
{\alpha}}/\im \delta_{(A_0, \Phi_0)}.
\]
The inclusion 
\[
\ker \delta^{
{*, \alpha}}_{(A_0, \Phi_0)} \hookrightarrow T_{(A_0, \Phi_0)}\Con_{l,
{\alpha}}
\]
also induces a linear isomorphism
\[
\ker \delta^{
{*, \alpha}}_{(A_0, \Phi_0)} \cong T_{(A_0, \Phi_0)}\Con_{l,
{\alpha}}/\im \delta_{(A_0, \Phi_0)}, 
\]
because a similar argument as the proof of Lemma \ref{decomposition} yields the $L^2_{\alpha}$ orthogonal decomposition
\[
T_{(A_0, \Phi_0)}\Con_{l,
{\alpha}}=\im \delta_{(A_0, \Phi_0)}\oplus \ker \delta^{
{*, \alpha}}_{(A_0, \Phi_0)} .
\]
For $(b, \psi) \in \ker \delta^{
{*, \alpha}}_{(A_0, \Phi_0)}$, let $(a, \phi) \in i\ker d^{
{*, \alpha}}\oplus L^2_{l,
{\alpha}}(S^+)$ be the element which corresponds to $(b, \psi)$ under the composition 
\[
\ker \delta^{
{*, \alpha}}_{(A_0, \Phi_0)} \cong T_{(A_0, \Phi_0)}\Con_{l,
{\alpha}}/\im \delta_{(A_0, \Phi_0)}\cong i\ker d^{
{*, \alpha}}\oplus L^2_{l,
{\alpha}}(S^+).
\]
We denote by $\chi=\chi(b, \psi) \in L^2_{k+1,
{\alpha}} (i\Lambda^0)$ the unique element that satisfies $(a, \phi)-(b, \psi)=\delta_{(A_0, \Phi_0)}\chi$.
We claim that the operator 
\[
\begin{split}
K: \ker \delta^{
{*, \alpha}}_{(A_0, \Phi_0)}&\to L^2_{l-1,
{\alpha}}(i\Lambda^+\oplus S^-)\\
(b, \psi)&\mapsto (0, \chi D^+_{A_0}\Phi_0)
\end{split}
\]
is compact.
Since $\delta_{(A_0, \Phi_0)}: L^2_{l+1,
{\alpha}}(i\Lambda^0)\to \im \delta_{(A_0, \Phi_0)}$ is a continuous linear isomorphim, 
its inverse is continuous by the open mapping theorem.
Thus, 
\[
\begin{split}
\ker \delta^{
{*, \alpha}}_{(A_0, \Phi_0)}&\to L^2_{l+1,
{\alpha}}(i\Lambda^0)\\
(b, \psi)&\mapsto \chi=\delta^{-1}_{(A_0, \Phi_0)}((a, \phi)-( b, \psi))
\end{split}
\]
is continuous.
Since $D^+_{A_0}\Phi_0$ is zero on $X^+\setminus X$,  the multiplication result of Lemma \ref{multiplication and compact embedding} implies that  the operator $K$ factors through the compact inclusion $L^2_{l+1, 2\alpha}(i\Lambda^+\oplus S^-)
 \hookrightarrow L^2_{l-1,
{\alpha}}(i\Lambda^+\oplus S^-)
$, for example.
 This shows that $K$ is compact  by the compactness result of Lemma \ref{multiplication and compact embedding}.
 \par
 Now, we can see that the following diagram commutes:
\[
  \begin{CD}
     (b, \psi)\in\ker \delta^{
{*, \alpha}}_{(A_0, \Phi_0)}@>{\cD_{(A_0, \Phi_0)}\F+K}>> L^2_{l-1,
{\alpha}}(i\Lambda^{+}\oplus S^-) \\
  @V{\cong}VV    @\vert\\
     (a, \phi)\in \ker d^{
{*, \alpha}} \oplus L^2_{l,
{\alpha}}(S^+)   @>\cD_{(A_0, \Phi_0)}\F>>  L^2_{l-1,
{\alpha}}(i\Lambda^{+}\oplus S^-)
  \end{CD}
\]
Indeed, since $(a, \phi)=\delta_{(A_0, \Phi_0)} \chi+(b, \psi)=(b-d\chi, \psi+\chi\Phi_0)$, we have
\[
\begin{split}
\cD_{(A_0, \Phi_0)}\F(a, \phi)
&=
\begin{bmatrix}
d^+(b-d\chi)-\rho^{-1}(\Phi_0(\psi+\chi\Phi_0)^*+(\psi+\chi\Phi_0)\Phi^*_0)_0\\
D^+_{A_0}(\psi+\chi\Phi_0)+\rho(b-d\chi)\Phi_0
\end{bmatrix}
\\
&=
\begin{bmatrix}
d^+b-\rho^{-1}(\Phi_0\psi^*+\psi\Phi_0)_0\\
D^+_{A_0}\psi+\rho(b)\Phi_0
\end{bmatrix}
+
\begin{bmatrix}
0\\
\chi D^+_{A_0}\Phi_0
\end{bmatrix}\\
&=\cD_{(A_0, \Phi_0)}\F(b, \psi)+K(b, \psi).
\end{split}
\]
This means that $L$ is also Fredholm and has the same index as $L'$.
\end{proof}
\subsection{Finite dimensional approximation}
Now, we will carry out the finite dimensional approximation and construct our Bauer-Furuta type invariant. 
We will follow the construction of \cite{Fur01}\cite{BF04}.
Now we will see that the Seiberg--Witten map with global slice can be written as the sum of the linear Fredholm map $L$ and a compact quadratic map $C$.
Define
\[
C: V\to W
\]
by 
\[
C
\begin{bmatrix}
a\\
\phi
\end{bmatrix}
=
\begin{bmatrix}
-\rho^{-1}(\phi\phi^*)_0\\
\rho(a)\phi\\
0
\end{bmatrix}, 
\]
Then, for $(A, \Phi)=(A_0+a, \Phi_0+\phi)\in \Con_{l,
{\alpha}}$, we have
\[
\begin{split}
&\begin{bmatrix}
\F(A, \Phi)\\
d^{
{*, \alpha}}a
\end{bmatrix}
\\
=&
\begin{bmatrix}
d^+a-\rho^{-1}(\Phi_0\phi^*+\phi\Phi^*_0)_0\\
D^+_{A_0}\phi+\rho(a)\Phi_0\\
d^{
{*, \alpha}}a
\end{bmatrix}
+
\begin{bmatrix}
-\rho^{-1}(\phi\phi^*)_0\\
\rho(a)\phi\\
0
\end{bmatrix}
+
\begin{bmatrix}
\frac{1}{2}F^+_{A^t_0}-\rho^{-1}(\Phi_0\Phi^*_0)_0\\
D^+_{A_0}\Phi_0\\
0
\end{bmatrix}
\\
=&
\begin{bmatrix}
\cD_{(A_0, \Phi_0)}\mathcal{F}(a, \phi)\\
d^{
{*, \alpha}}a
\end{bmatrix}
+
\begin{bmatrix}
-\rho^{-1}(\phi\phi^*)_0\\
\rho(a)\phi\\
0
\end{bmatrix}
+
\begin{bmatrix}
\frac{1}{2}F^+_{A^t_0}-\rho^{-1}(\Phi_0\Phi^*_0)_0\\
D^+_{A_0}\Phi_0\\
0
\end{bmatrix}\\
=&
L
\begin{bmatrix}
a\\
\phi
\end{bmatrix}
+
C
\begin{bmatrix}
a\\
\phi
\end{bmatrix}
+
\begin{bmatrix}
\frac{1}{2}F^+_{A^t_0}-\rho^{-1}(\Phi_0\Phi^*_0)_0\\
D^+_{A_0}\Phi_0\\
0
\end{bmatrix}.
\end{split}
\]
Thus, if we write $w_0=(0, -D^+_{A_0}\Phi_0, 0) \in W$
we have
\[
\begin{split}
&(A_0, \Phi_0)+(L+C)^{-1}(w_0)\\
=&\{(A, \Phi) \in \Con_{l,
{\alpha}} | (A, \Phi) \text{ solves the Seiberg--Witten equation and satisfies } d^{
{*, \alpha}}(A-A_0)=0\}
\end{split}
\]
as subsets in $\Con_{l,
{\alpha}}$.
\par
The first step of the finite dimensional approximation is the compactness of the moduli space.
The following compactness result is proved in \cite{KM97}3(iii)(See also \cite{MR06} Proposition 2.2.6). 
\begin{prop}(\cite{KM97}).\,
Suppose $l\in \Z^{\geq 4}$.
Then,  there exists $\alpha_0>0$ such that for any $0<\alpha<\alpha_0$ the following holds.
For any sequence of solutions $(A^{(i}, \Phi^{(i)}) \subset \Con_{l,
{\alpha}}$, there is a sequence of gauge transformations $u^{(i)} \subset \G_{l+1,
{\alpha}}$ such that after passing to a subsequence, the transformed solutions converge in $\Con_{l,
{\alpha}}$: that is, the moduli space
\[
\mathcal{M}_{l,
{\alpha}}=\{(A, \Phi) \in \Con_{l,
{\alpha}} | (A, \Phi) \text{ solves the  equation} \eqref{SWeq}\}/\G_{l+1,
{\alpha}}
\]
is compact.
\qed
\end{prop}
\begin{prop}\label{Step1}
For sufficiently small $\alpha>0$, 
There exists $R>0$ such that $(L+C)^{-1}(w_0) \subset \text{Int} B^{V}(0, R)$, where $B^{V}(0, R)$ is the ball of radius $R$ centered in $0$ in $V$.
\end{prop}
\begin{proof}
By the global slice theorem \ref{global slice}, $(L+C)^{-1}(w_0)$ is homeomorphic to the moduli space $\mathcal{M}_{l,
{\alpha}}$, so the conclusion follows from the proposition above.
\end{proof}
Before turning to the next step of the finite dimensional approximation, we will prepare
the following lemma which is used as a counterpart of the elliptic estimate.
\begin{lem}
Let $(V, \|\cdot \|_V)$, $(W, \|\cdot \|_W)$ be Banach spaces and $L: V \to W$ be a continuous linear Fredholm operator.
Let $n$ be a norm on $V$  which is weaker than $\|\cdot\|_V$).
Then, there exists a positive constant $C$ such that for any $v \in V$, 
\[
\|v\|_V \leq C(\|Lv\|_W+n(v))
\] 
holds.
\end{lem}
\begin{proof}
It is enough to prove the following claim:
\begin{claim}
Let $(X, \|\cdot\|_X)$, $(Y, \|\cdot\|_Y)$ be norm spaces and suppose $Y$ is finite dimensional.
Consider two norms on $V=X\oplus Y$: one is the norm as direct sum $\|x\|_X+\|y\|_Y$, and the other is a norm $n$ such that there exists a positive constant $C_0$ such that for any $(x, y)\in V$, $n(x, y)\leq C_0(\|x\|_X+\|y\|_Y)$.
 Then, there exists a positive constant $C$ such that for any $(x, y) \in V$, 
 \[
 \|x\|_X+\|y\|_Y\leq C(\|x\|_X+n(x, y))
\]
holds.
\end{claim}
The correspondence between the lemma and claim is as follows.
$Y=\ker L$, $X$ is a complement of $Y$ in $X$, $W=X\oplus \operatorname{cok} L$, and $L: (x, y)\mapsto x$, where $x \in X, \, y \in Y$.
Here, we used the fact that the image of $ L $ is closed and $L: X\to  \im L$ is a continuous linear isomorphism by the open mapping theorem, so we may assume it is the identity map.

\par
\underline{Proof of Claim.}\,

Since $Y$ is finite dimensional, there exists a positive constant $c$ such that for any $y \in Y$, 
\[
c^{-1}\|y\|_Y\leq n(0, y)\leq c\|y\|_Y
\]
holds.
Thus, we have
\[
 \|x\|_X+\|y\|_Y\leq \|x\|_X+cn(0, y)\leq \|x\|_X+c(n(x, y)+n(x, 0))\leq (1+cC_0)\|x\|_X+cn(x, y)
\]
and this means that the claim holds for $C=\max\{1+cC_0, c\}$
\end{proof}
The following is the second step of the finite dimensional approximation.
\begin{prop}\label{Step2}
Let $R>0$ be the constant given by Proposition \ref{Step1}, there exists a sufficiently small constant $\epsilon >0$ such that $(L+C)(S^V(0, R))\cap B^W(w_0, \epsilon)=\emptyset$, where $S^V(0, R)$ is the sphere of radious $R$ centerd at $0$ at $V$ and $B^W(w_0, \epsilon)$ is the closed ball of radious $\epsilon$ centered at $w_0$ in $W$.
\end{prop}
\begin{proof}
We suppose the contrary. Then we can pick a sequence $(v_n) \subset S^V(0, R)$ such that $(L+C)(v_n) \to w_0$ as $n\to \infty$.
Since $L$ is Fredholm, the above lemma implies that there exists a positive constant $C_0$ such that for any pair $(m, n)$, we have
\[
\|v_n-v_m\|_{L^2_{l,
{\alpha}}}\leq C_0(\|L(v_n-v_m)\|_{L^2_{l-1,
{\alpha}}}+\|v_n-v_m\|_{L^2}).
\]
Since $C$ is a compact operator, we can take a subsequence such that $C(v_n)$ converge, and hence $L(v_n)$ converge.
Thus,  the right-hand side converges to zero after passing to a subsequence by the compact embedding of Lemma \ref{multiplication and compact embedding}.
Thus, $v_n$ is Cauchy and convergent to some $v_\infty$ in $V$ in this subsequence.
Then, $v_\infty \in S^V(0, R)$ and $(L+C)(v_\infty)=w_0$.
This contradicts the condition on $R$.
\end{proof}
Now we fix a pair of sequences of finite dimensional subspaces 
\[
\ker L \subset V_0 \subset V_1 \subset \cdots \subset V
\]
\[
(\im L)^\perp \subset W_0 \subset W_1 \subset \cdots \subset W
\]
(where $(\im L)^\perp$ is the orthogonal complement of $\im L$ with respect to the inner product of $W$)
satisfying the following two conditions
\begin{enumerate}
\item
$V_n=L^{-1}(W_n)$.
\item
$\lim_{n\to \infty} P_n w=w$ for any $w \in W$, where $P_n: W\to W_n$ is the orthogonal projection with respect to the inner product of $W$.
\end{enumerate}
\begin{prop}\label{Step3}
 Let $R, \epsilon$ be the constants given by Proposition \ref{Step1}, Proposition \ref{Step2}.
Then, there exists a  natural number $N\in \Z^{\geq 0}$ such that for any $n \in \Z^{\geq N}$, $v \in S^V(0, R)$, we have $\|(1-P_n)(C(v)+w_0)\|_W<\epsilon$.
\end{prop}
\begin{proof}
We can ignore the term $w_0$ in the statement due to the condition (2).
We suppose the contrary.
Then we can pick a sequence $(v_n) \subset S^V(0, R)$ such that for any $n \in \Z^{\geq 0}$, $\|(1-P_n)C(v_n)\|_W\geq \epsilon$ holds.
Since $(v_n)$ is bounded, $v_n$ converges weakly to some $v_\infty$ in $V$ after passing to a subsequence.
Then, $C(v_n)$ converges strongly to $C(v_\infty)$ in $W$.
By the condition (2) for $\{V_n\}, \{W_n\}$, $(1-P_n)C(v_\infty)\to 0$ as $n\to \infty$.
Thus, there exists $N_0 \in \Z^{\geq 0}$ such that for any $n \in \Z^{\geq N_0}$, 
\[
\|(1-P_n)C(v_\infty)\|_W <\frac{\epsilon}{2}
\]
holds.
Since $P_{N_0}(C(v_n)-C(v_\infty))\to 0$ as $n \to \infty$, there exists $N \in \Z^{\geq N}$ such that
\[
\|P_{N_0}(C(v_n)-C(v_\infty))\|_W<\frac{\epsilon}{2}
\]
holds.
Thus we have
\[
\begin{split}
\|(1-P_N)C(v_N)\|_W&\leq \|(1-P_{N_0})C(v_N)\|_W\\
&\leq \|(1-P_{N_0})C(v_N)\|_W+\|P_{N_0}(C(v_N)-C(v_\infty))\|_W\\
&<\frac{\epsilon}{2}+\frac{\epsilon}{2}\\
&=\epsilon.
\end{split}
\]
This contradicts the hypothesis.
\end{proof}
The following proposition gives our Bauer--Furuta-type invariant:
\begin{prop}\label{construction}
For $N$ in Proposition \ref{Step3}, the following holds.
For any $n \in \Z^{\geq N}$, $v \in S^{V}(0, R)$ we have $(L+P_nC)(v)\neq P_nw_0$ and thus the map $L+P_nC: S^{V_n}(0, R) \to W_n\setminus \{P_nw_0\}$ is well-defined.
Its stable homotopy class $\Psi(X, \xi, \s)\in \pi^{st}_{d(\s)}(S^0)$ (up to sign) depends only on $X$, $\s \in \spinc(X, \xi)$, and the isotopy class of $\xi$.
Here, $d(\s)=\Index L$ is the same as the virtual dimension of the moduli space for  Kronheimer--Mrowka's invariant $\mathfrak{m}(X, \xi, \s)$.
\end{prop}
\begin{proof}
For $v \in S^{V}(0, R)$, we have
\[
\|(L+P_nC)-P_nw_0\|_W\geq \|(L+C)(v)-w_0\|_W-\|(1-P_n)(C(v)+w_0)\|_W>\epsilon-\epsilon=0.
\]
Thus, $(L+P_nC)(v)\neq P_nw_0$.
In order to obtain a based map, one can take the cone of this map and collapse the boundary.
(In general, for a boundary preserving continuous map between disks
\[
f: (D^n, \partial D^n)\to (D^m, \partial D^m), 
\]
the map obtained by collapsing the boundaries
\[
D^n/\partial D^n\to D^m/ \partial D^m
\]
is homotopic to the suspension of the map obtained by restricting to the boundaries
\[
\partial D^n\to \partial D^m, 
\]
as one can show easily by constructing explicit homotopy.)
The proof of the independence is standard.
\end{proof}
\begin{rem}
Unlike the usual Bauer-Furuta invariant for closed 4-manifolds, this stable map is not $S^1$ equivariant because $1$ is the only constant gauge transformation in the present setting.
\end{rem}

\section{Properties}
\subsection{Recovery of Kronheimer--Mrowka's invariant}
The following theorem is contact-boundary version of the fact that the Bauer--Furuta invariant recovers the Seiberg--Witten invariant.
\begin{thm}\label{recover}
When $H^1(X, \partial X; \R)=0$ and $d(\s)=0$, the mapping degree of our invariant $\Psi(X, \xi, \s)$ equals Kronheimer--Mrowka's invariant $\mathfrak{m}(X, \xi, \s)$ up to sign. 
\end{thm}
\begin{proof}
Take $n$ as in Proposition \ref{construction}.
We compare three equations.
\begin{enumerate}
\item
$L+C: V\to W$
\item
$L+P_nC: V\to W$
\item
$L+P_nC: V_n \to W_n$
\end{enumerate}
For a suitable perturbation $\eta_0$ of (1), $(L+C)^{-1}(\eta_0)$ is cobordant to the moduli space for Kronheimer--Mrowka's invariant $\mathfrak{m}(X, \xi, \s)$ because of the global slice theorem \ref{global slice}.
Thus it is enough to prove the following claims.
\begin{claim}
\begin{enumerate}
\item
For $\eta_1 \in W_n$, the inverse images of (2) and (3) are the same.
\item
The regular value $\eta_1 \in W_n$ of (3) is also a regular value for (2).
\item
Take a perturbation $\eta_0$ of (1) by standard argument and take a regular value of (3) by (finite dimensional) Sard's theorem.
Consider one parameter family of equations $L+C_t=L+(1-t)C+tP_nC\, (t\in [0, 1])$ which linearly joining (1) to (2).
Then, there exists a path $\eta_\bullet: [0, 1] \to W$ from $\eta_0$ to $\eta_1$ such that 
\[
\bigcup_{t\in [0, 1]}(L+C_t)^{-1}(\eta_t) \subset [0, 1]\times V
\]
is an orientable cobordism from $(L+C)^{-1}(\eta_0)$ to $(L+P_nC)^{-1}(\eta_1)$.
\end{enumerate}
\end{claim}
\underline{Proof of Claim.}\,
Suppose $\eta_1 \in W_n$ and $(L+C)(v)=\eta_1$ for some $v \in V$. 
Then, 
\[
(1-P_n)Lv=Lv+P^2_nC(v)-P_n\eta_1=(L+P_nC)(v)-\eta_1=0.
\]
and thus $v\in V_n$. This proves the first claim.
\par  The second claim can be easily checked. The third claim follows from the standard argument.
\end{proof}
From the non-vanishing theorem of Kronheimer--Mrowka's invariant for weak symplectic filling \cite{KM97} Theorem1.1, the following statement follows.
\begin{cor}\label{nonvanishing}
Let $(X, \omega)$ be a weak symplectic filling of $(\partial X, \xi)$ with $H^1(X, \partial X; \R)=0$ (which includes all Stein fillings). 
Then for $\s_\omega\in \spinc(X, \xi)$ canonically determined by symplectic structure $\omega$, $\Psi(X, \xi, \s_\omega)$ is a generator of $\pi^{st}_0(S^0)\cong \Z$.
\qed
\end{cor}
\subsection{Connected sum formula}
For a closed 4-manifold $X$ with $b_1(X)=0$ equipped with a $\spinc$ structure $\s$, we denote its 
Bauer--Furuta invariant by 
\[
\tilde{\Psi}(X, \s) \in \pi^{st}_{d(\s)+1}(S^0)
\]
forgetting the $S^1$ action, where
\[
d(\s)=\frac{1}{4}(c^2_1(\s)-2\chi(X)-3\sigma(X)).
\]
We shall prove the following connected sum formula.
\begin{thm}\label{connected sum formula}
Let $(X_1, \s_1)$ be a closed oriented connected 4-manifold with $b_1(X_1)=0$ equipped with a $\spinc$ structure and let $(X_2, \xi, \s_2)$ be a compact oriented connected 4-manifold  with contact boundary equipped with $\s_2 \in \spinc(X_1, \xi)$. In addition, suppose $H^1(X_1, \partial X_1; \R)=0$ in order to define our invariant of  $(X_2, \xi, \s_2)$. Then, we have
\[
\Psi(X_1\#X_2, \xi , \s_1\#\s_1)=\tilde{\Psi}(X_1, \s_1)\wedge \Psi(X_2, \xi, \s_2)\in \pi^{st}_{d(\s_1)+d(\s_2)+1}(S^0).
\]
\end{thm}
\begin{proof}
This follows from a more general result Theorem \ref{generalized connected sum}, which we will prove later. 
\end{proof}

Let $n$ be a natural number (the result will be trivial unless $n\geq  3$) and 
$\tau: \{1, \dots, n\}\to \{1, \dots, n\}$ be an even permutation.
Let $X_{1, -}$ and $X_{i, \pm} (i=2, \dots, n)$ be compact oriented connected 4-manifolds with $b_1=0$ equipped with $\spinc$ structures each of which has $S^3$ as the only component of the boundary.
Let $\hat{X}$ be a compact oriented connected 4-manifold such that $\partial \hat{X}$ supports a contact structure, and $H^1(\hat{X}, \partial\hat{X}; \R)=0$.
Let $X_{1, +}$ be a noncompact manifold with $\partial X_{1, +}=S^3$ obtained from $\hat{X}$ by attaching the almost K\"ahler cone and removing a small 4-ball.
For $L\geq 0$, consider
\[
X_i(L)=X_{i, -}\cup_{S^3} ([-L, L]\times S^3)\cup_{S^3} X_{i, +}
\]
\[
X^\tau_i(L)=X_{i, -}\cup_{S^3} ([-L, L]\times S^3)\cup_{S^3} X_{\tau(i), +}.
\]
with the product metric of the round metric of $S^3$ and the standard metric on $[-L, L]$ on each neck $[-L, L]\times S^3$.
We denote their disjoint union by
\[
X(L)=\coprod^n_{i=1}X_i(L)
\]
and
\[
X^\tau(L)=\coprod^n_{i=1}X^\tau_i(L).
\]

As in the construction of our Bauer-Furuta-type invariant in the previous section, pick a based configuration $(A_0, \Phi_0)$ on $X_{1, +}$ and fix a smooth extension to $X(L)$ and $X^\tau(L)$ such that $A_0$ is flat and $\Phi_0$ is zero on each neck $[-L, L]\times S^3$ (In particular the Seiberg--Witten equation is the same as the usual one for closed manifolds on each neck).
Pick a smooth extension $\sigma$ of the $\R^{\geq 1}$ coordinate on the almost K\"ahler cone to $X_{1, +}$ such that its support is contained in the interior and extend it to $X(L), X^\tau(L)$ as zero outside $X_{1, +}$.
For $l \in \Z^{\geq 4}$ and $\alpha>0$, we define   $V(X(L))$ to be $L^2_{l,
{\alpha}}(i\Lambda^1\oplus S^+)$ and 
$W(X(L))$ to be the $L^2_\alpha$ orthogonal complement  of the $n-1$ dimensional space consists of constant $0$-forms on $X_i(L)\, (i=2, \dots, n)$ in $L^2_{l-1,
{\alpha}}(i\Lambda^0\oplus i\Lambda^+ \oplus S^-)$.
Define $V(X^\tau(L)), W(X^\tau(L))$ similarly.
As in the previous section, we denote the Seiberg--Witten map with the global slice by
\[
\mu: =L+C: V(X(L)) \to W(X(L))
\]
\[
\mu^\tau=L^\tau+C^\tau: V(X^\tau(L)) \to W(X^\tau(L)).
\]
For each $L\geq 0$, the finite dimensional approximation of $\mu$ gives 
\[
[\mu]=\Psi(X_{1, -}\cup_{S^3}X_{1, +})\wedge \tilde{\Psi}(X_{2, -}\cup_{S^3}X_{2, +})\wedge \cdots \wedge \tilde{\Psi}(X_{n, -}\cup_{S^3}X_{n, +}) \in \pi^{st}_{d(\s_1)+\dots + d(\s_n)+n-1}(S^0).
\]
The case for $\mu^\tau$ is similar.
As in \cite{Bau04}, by fixing a function $\psi: [-1, 1] \to [0, 1]$ which gives a cut off function on the necks and a path $\gamma$ in $SO(n)$ from the $1_{SO(n)}$ and $\tau$, an identifications
\[
V(X(L)) \xrightarrow{\cong} V(X^\tau(L)), \quad W(X(L))\xrightarrow{\cong} W(X^\tau(L))
\]
are determined.
The following is the generalized statement of the connected sum formula.
\begin{thm}\label{generalized connected sum}
In the above setting, we have
\[
[\mu]=[\mu^\tau] \in \pi^{st}_{d(\s_1)+\cdots+d(\s_n)+n-1}(S^0)
 \]
under the identification above.
(In particular,  Theorem \ref{connected sum formula} follows by considering the permutation between
\[
(X_1\# X_2)\amalg (S^4\# S^4)\amalg (S^4\# S^4)
\]
and
\[
(X_1 \# S^4)\amalg (S^4\# X_2)\amalg (S^4\# S^4).
\]
and Proposition2.3 of \cite{Bau04}
)
\end{thm}
\begin{proof}
This can be shown by the similar method as Bauer's original proof of Theorem 2.1 of \cite{Bau04}.
We will explain the main steps following \cite{Bau04}.
For
$1\leq R\leq L$, let $\beta_R$ be a cut off function $X(L)\to [0, 1]$ with  $ \beta_R\equiv 0$ on $X(L)\setminus ([-R+1, R-1]\times S^3)$,  $\beta_R\equiv 1$ on $[-R, R]\times S^3$, and   depends only on $[-L, L]$-coordinate. 
Set
$\beta_{t, R}=(1-t)+t\beta_{R}$.
Consider the following three types of deformations.
\begin{enumerate}
\item
\[
\mu^{(1)}_t=L+C^{(1)}_t, \quad t \in [0, 1]
\]
where
\[
C^{(1)}_t
\begin{bmatrix}
a\\
\phi
\end{bmatrix}
=\begin{bmatrix}
-\beta_{L, t}\rho^{-1}(\phi\phi^*)_0\\
\rho(a)\phi\\
0
\end{bmatrix}
\]

\item
\[
\mu^{(2)}_t=L^{(2)}_t+C^{(1)}_1, \quad t \in [0, 1]
\]
where
\[
L^{(2)}_t
\begin{bmatrix}
a\\
\phi
\end{bmatrix}
=\begin{bmatrix}
d^+a-\rho^{-1}(\Phi_0 \phi^*+\phi\Phi^*_0)_0\\
D^+_{A_0}\phi+\rho(\beta_{2, t}a)\Phi_0\\
d^{*, \alpha}a
\end{bmatrix}
\]
\item
\[
\mu^{(3)}_t=\mu^{\tau (2)}_1+ d\log (V_t)
\]
where, $V_t$ is
\[
(a, \phi)\mapsto (\psi \circ t\gamma)(a, \phi)
\]
on the neck and extended in the obvious way on  $X(T)$.
\end{enumerate}
The following statements correspond to Lemma 3.2, 3.3, 3.4, 3.5 of \cite{Bau04} and can be shown in a similar way.
\begin{claim}
\begin{enumerate}
\item
$(\mu^{(1)}_t)^{-1}(w_0)$ is bounded for all $t \in [0, 1]$.
\item
There exist  constants $L_1$, $R_*$ such that if $L\geq L_1$,  the following holds on $X(L)$.
Any $(a, \phi) \in (\mu^{(1)}_1)^{-1}(w_0)$ satisfies $\|(a, \phi)\|_{L^2_{l, \alpha}}\leq  R_* $.
\item
There exist $L_1$, such that if $L\geq L_1$, the following holds on $X(L)$.
For any $(a, \phi) \in \cup_{t \in [0, 1]}(\mu^{(2)}_t)^{-1}(w_0)$ satisfying $\|(a, \phi)\|^2_{L^2_{l, \alpha}}\leq 2R_* $ in fact satisfies $\|(a, \phi)\|_{L^2_{l, \alpha}}\leq R_* $.
\item
There exist $L_2\geq L_1$ such that if $L\geq L_2$, then the following holds on $X(L)$.
For any $(a, \phi) \in \cup_{t \in [0, 1]}(\mu^{(3)}_t)^{-1}(w_0)$ satisfying $\|(a, \phi)\|^2_{L^2_{l, \alpha}}\leq 2R_* $ in fact satisfies $\|(a, \phi)\|_{L^2_{l, \alpha}}\leq R_* $.
\end{enumerate}
\end{claim}
By composing  homotopoies obtained from these deformations on $X(L)$ and corresponding homotopies on $X^\tau(L)$, the conclusion follows.
\end{proof}
\subsection{Invariance under the connected sum of rational homology spheres}
Combining the connected sum formula Theorem \ref{connected sum formula} and \cite{Bau04} Proposition 2.3,
we see that our invariant is invariant under the connected sum of rational homology spheres.
\begin{thm}\label{QHS4}
Let $(X, \xi)$ be a compact oriented connected 4-manifold with contact boundary satisfying $H^1(X, \partial X; \R)=0$ and let $M$ be a rational homology $4$ sphere.
Then, for any $\spinc$ structure $\s_0$ on $M$ and for $\s_1\in \spinc(X, \xi)$, we have
\[
\Psi(M\# X, \xi, \s_0\#\s_1)=\Psi(X, \xi, \s_1).
\] 
\end{thm} 
\begin{proof}
By \cite{Bau04} Proposition 2.3, $\tilde{\Psi}(M, \s_0) \in \pi^{st}_0(S^0)$ is a generator, so the conclusion follows from the connected sum formula.
\end{proof}
The following corollary is given in \cite{Ger19} Proposition 5.3. 
\begin{cor}(\cite{Ger19})\,
Let $X$ be an oriented compact connected contractible $4$-manifold with $\partial X=S^3$. Let $\xi_{std}$ be the standard contact structure on $S^3$.
Then, for the unique element $\s\in \spinc(X,\xi_{std})$, we have
\[
\mathfrak{m}(X, \xi_{std}, \s)=\pm1.
\]
\end{cor}
\begin{proof}
$M=X\cup_{S^3} D^4$ is a homotopy $S^4$.
So, Theorem \ref{recover} and the above theorem implies the conclusion.
\end{proof}
\begin{rem}
It is well-known that the smooth 4-dimensional Poincar\'e conjecture is equivalent to the statement that any contractible compact oriented 4-manifold whose boundary is $S^3$ has a structure of a Stein filling of $(S^3, \xi_{std})$.
The result above means that it is impossible to disprove this conjecture by denying the existence of Stein structures on such manifolds via Kronheimer--Mrowka's invariant.
\end{rem}
\subsection{Examples}
Using the connected sum formula and \cite{Bau04} Propsition 4.4, we can construct examples of $(X, \xi, \s)$ with $\mathfrak{m}(X, \xi, \s)=0$ but $\Psi(X, \xi, \s)\neq0$.
\begin{thm} 
Suppose closed oriented 4-manifolds with $\spinc$ structures $(X_i, \s_i)\, (i=1, 2)$ both have the following property:
\begin{enumerate}
\item
$X_i$ has an almost complex structure and $\s_i$ is the $\spinc$ structure determined by it.
\item
$b_1(X_i)=0$
\item
$b^+(X_i)\equiv 3\, (\mathrm{mod}\, 4)$
\item
The Seiberg--Witten invariant $\mathfrak{m}(X_i, \s_i)$ is an odd integer.
\end{enumerate}
Let $(X_3, \omega)$ be a weak symplectic filling of $(\partial X_3, \xi)$ with $H^1(X_3, \partial X_3; \R)=0$ and take $\s_3=\s_{\omega} \in \spinc(X_3, \xi)$.
Then, 
\[
\Psi(X_1\#X_3, \xi, \s_1\#\s_3)\in \pi^{st}_1(S^0)\cong \Z/2
\]
and 
\[
\Psi(X_1\#X_2\#X_3, \xi, \s_1\#\s_2\#\s_3)\in \pi^{st}_2(S^0)\cong \Z/2
\]
are nontrivial. 
\end{thm}
\begin{proof}
\cite{Bau04} Proposition 4.4 tells that 
\[
\tilde{\Psi}(X_i, \s_i)\in \pi^{st}_1(S^0)\cong \Z/2
\]
is nontrivial
for $i=1, 2$.
By Corollary \ref{nonvanishing}, 
\[
\Psi(X_3, \xi, \s_3) \in \pi^{st}_0(S^0)
\]
is a generator.
Thus, the connected sum formula implies the conclusion.
\end{proof}
\begin{ex}
Elliptic surfaces $E(2n) \, (n=1, 2, \dots)$ equipped with the $\spinc$ structure determined by K\"ahler structures are examples of $(X_i, \s_i)\, (i=1, 2)$.
Any Stein fillings are examples of $(X_3, \xi, \s_3)$.
\end{ex}
On the other hand, for manifolds $X_1\#X_3$, $X_1\#X_2\#X_3$ in the above theorem, we can show that their Kronheimer--Mrowka's invariants are zero for {\bf all} $\spinc$ structures on them:
\begin{prop}\label{vanish}
Let $X_1$ be a closed oriented connected 4-manifold with $b_1(X_1)=0$ and $b^+(X_1)\geq 1$.
Let $(X_2, \xi)$ be a compact oriented connected 4-manifold with contact boundary.
Then, for any $\s \in \spinc(X_1\#X_2, \xi)$, $\mathfrak{m}(X_1\#X_2, \xi, \s)$ is zero. In particular, $X_1\#X_2$ never has a structure of a weak symplectic filling.
\end{prop}
\begin{proof}
Since $\spinc(X_1\#X_2, \xi)$ is a principal homogeneous space for \\$H^2(X_1\#X_2, \partial (X_1\#X_2))=H^2(X_1)\oplus H^2(X_2, \partial X_2)$, 
any $\s \in \spinc(X_1\#X_2, \xi)$ can be realized as a connected sum $\s=\s_1\#\s_2$, where $\s_1$ is a $\spinc$ structure on $X_1$ and $\s_2 \in \spinc(X_2, \xi)$.
Then, we have
\[
d(\s)=d(\s_1)+d(\s_2)+1.
\]
If $d(\s)\neq 0$, $\mathfrak{m}(X_1\#X_2, \xi, \s)=0$ by definition.
If $d(\s)=0$, either $d(\s_1)<0$ or $d(\s_2)<0$ holds.
So, by perturbing the Seiberg--Witten equation, we can make one of the moduli spaces empty (here, we use the assumption $b^+(X_1)\geq 1$) and thus either $\tilde{\Psi}(X_1, \s_1)$ or $\Psi(X_2, \xi, \s_2)$ is zero.
This implies that $\mathfrak{m}(X_1\#X_2, \xi, \s)=0$.
\end{proof}
\subsection{An application}
As an application of the connected sum formula, we can show the following results on the existence of a connected sum decomposition $X=X_1\#X_2$ for a 4-manifold with contact boundary $X$.
This result can be seen as a contact-boundary version of \cite{Bau04} Corollary 1.2, Corollary1.3 for closed manifolds.
\begin{thm}
Let $(X, \xi)$ be a compact oriented connected 4-manifold with contact boundary satisfying $H^1(X, \partial X; \R)=0$.
Suppose there exists $\s \in \spinc(X, \xi)$ such that $\Psi(X, \xi, \s)\neq 0$.
Then, if $X$ can be decomposed as a connected sum $X_1\#X_2$ for a closed 4-manifold $X_1$ and a 4-manifold with contact boundary $X_2$, the following holds.
\begin{enumerate}
\item
Suppose $d(\s)=0$. Then $X_1$ is negative definite.
\item
Suppose $d(\s)=1$. Then either of the following holds.
\begin{enumerate}
\item
$b^+(X_1)=0$ 
\item
$b^+(X_1)\equiv 3\, (\mathrm{mod} 4)$ and there exist a $\spinc$ structure $\s_1$ on $X_1$ and $\s_2 \in \spinc(X_2, \xi)$ such that $\mathfrak{m}(X_1, \s_1)$ and $\mathfrak{m}(X_2, \xi, \s_2)$ are both odd.
\end{enumerate}
\item
Suppose $d(\s)=2$. Then, $b^+(X_1)\not\equiv 1 (\mathrm{mod} 4)$ holds.
Furthermore, if $b^+(X_1)\neq 0$, either of the following holds.
\begin{enumerate}
\item
$b^+(X_1)\equiv 3 \,(\mathrm{mod} 4)$ and there exists a $\spinc$ structure $\s_1$ on $X_1$ such that $\mathfrak{m}(X_1, \s_1)$ is odd.
\item
$b^+(X_1)$ is even and there exists $\s_2 \in \spinc(X_2, \xi, \s_2)$ such that $\mathfrak{m}(X_2, \xi, \s_2)$ is odd.
\end{enumerate}
\end{enumerate}
\end{thm}
\begin{proof}
(1) follows from Proposition \ref{vanish}  and Theorem \ref{recover}. 
In order to show (2) and (3), we use the following results of Bauer.
\begin{prop}\label{Bauer}(\cite{Bau04})\,
Let $X$ be a closed oriented connected 4-manifold with $b_1(X)=0$ and $\s$ be a $\spinc$ stucture on it. Then the following holds.
\begin{enumerate}
\item[(i)]
If $d(\s)=-1$ and $\tilde{\Psi}(X, \s)\in \pi^{st}_0(S^0)\cong \Z$ is odd, then $b^+(X)=0$.
\item[(ii)]
If $d(\s)=0$ and $\tilde{\Psi}(X, \s) \neq 0 \in \pi^{st}_1(S^0) \cong \Z/2$, then $b^+(X)\equiv 3 \, (\mathrm{mod} 4)$ and $\mathfrak{m}(X, \s)$ is odd.
\end{enumerate}
\end{prop}
\begin{proof}
(i)
This argument is quoted from Corollaly 1.3  of \cite{Bau04}.
In the case $d(\s)=-1$, the Bauer-Furuta invariant is represented by an $S^1$ equivariant map
\[
f: (\R^n\oplus \C^{m+\frac{b^+(X)}{2}})^+\to (\R^{n+b^+(X)}\oplus \C^{m})^+.
\]
If $b^+(X)\neq 0$, then the non-equivarinat degree of such a map must be zero by equivariant homotopy theory.
\par
(ii)
 follows Proposition 4.4  of \cite{Bau04}.
Notice that although the assumption written there is that the $\spinc$ structure    $\s$ comes from an almost complex structure,  this is equivalent to $d(\s)=0$, as is well known.
One argument to show this is to use the fact that $d(\s)$ is equal to the Euler number of the positive spinor bundle  $S^+$ of $\s$ (See Lemma 28.2.4 of \cite{KM07}, for example.) and the correspondence between nowhere vanishing sections of $S^+$ and almost complex structures (See Lemma 2.1 of \cite{KM97}.).
\end{proof}
In the both cases (2)(3), there exist a $\spinc$ structure $\s_1$ on $X_1$ and $\s_2 \in \spinc(X_2, \xi)$ such that $\s=\s_1\#\s_2$.
The connected sum formula implies
\[
\Psi(X, \xi, \s)=\tilde{\Psi}(X_1, \s_1)\wedge \Psi(X_2, \xi, \s_2) \in \pi^{st}_{d(\s)}(S^0)
\]
and $d(\s)=d(\s_1)+d(\s_2)+1$.
Thus, $\tilde{\Psi}(X_1, \s_1)$ and $\Psi(X_2, \xi, \s_2)$ are both nontrivial and in particular the degree of the stable homotopy groups which they belong, $d(\s_1)+1$, $d(\s_2)$,  must be non-negative.
\par
Now we will show (2). In this case, $(d(\s_1), d(\s_2))= (-1, 1), (0,0)$.
\begin{itemize}
\item
In the case $(d(\s_1), d(\s_2))= (-1, 1)$, Proposition \ref{Bauer} implies $b^+(X_1)=0$.
\item
In the case $(d(\s_1), d(\s_2))= (0,0)$, Proposition \ref{Bauer} implies 
$b^+(X_1)\equiv 3 \, (\mathrm{mod} 4)$ and $\mathfrak{m}(X_1, \s_1)$ is odd and the fact that $0 \neq \tilde{\Psi}(X_1, \s_1)\wedge \Psi(X_2, \xi, \s_2) \in \pi^{st}_{1}(S^0)\cong \Z/2$ and $d(\s_2)=0$ implies $\mathfrak{m}(X_2, \xi, \s_2)$ is odd.

\end{itemize}
This proves (2).
\par
Finally, we turn to (3). In this case, $(d(\s_1), d(\s_2))=(-1, 2),(0, 1), (1, 0)$.
\begin{itemize}
\item
In the case $(d(\s_1), d(\s_2))=(-1, 2)$, Proposition \ref{Bauer} implies $b^+(X_1)=0$.
\item
In the case $(d(\s_1), d(\s_2))=(0, 1)$, Proposition \ref{Bauer} implies 
$b^+(X_1)\equiv 3 \, (\mathrm{mod} 4)$ and $\mathfrak{m}(X_1, \s_1)$ is odd.
\item
In the case $(d(\s_1), d(\s_2))=(1, 0)$, $b^+(X_1)$ is even since $d(\s_1)\equiv 1+b^+(X_1)\, (\mathrm{mod} 2)$.
The fact that $\mathfrak{m}(X_2, \xi, \s_2)$ is odd follows from the same reason as in the latter case of (2).
\end{itemize}
This proves (3).
\end{proof}
\bibliographystyle{jplain}
\bibliography{mron}
\end{document}